\documentclass[a4paper]{amsart}

\newcommand{\remove}[1]{ }

\newtheorem{theorem}{Theorem}[section]

\newtheorem{lemma}[theorem]{Lemma}

\theoremstyle{definition}

\theoremstyle{remark}

\numberwithin{equation}{section}

\begin{document}
\title[Ingham--Beurling type estimates]{A vectorial Ingham--Beurling theorem}
\author{Alia Barhoumi}
\address{D\'epartement de Math\'ematique, Facult\'e des Sciences de Monastir,\\
5019, Monastir,Tunisie}
\email{Alia.Barhoumi@isimm.rnu.tn}
\author{Vilmos Komornik}
\address{D\'epartement de Math\'ematique,
         Universit\'e Louis Pasteur,
         7 rue Ren\'e Descartes, 67084 Strasbourg Cedex, France}
\email{komornik@math.u-strasbg.fr}
\author{Michel Mehrenberger}
\address{D\'epartement de Math\'ematique,
         Universit\'e Louis Pasteur,
         7 rue Ren\'e Descartes, 67084 Strasbourg Cedex, France}
\email{mehrenbe@math.u-strasbg.fr}

\begin{abstract}
Baiocchi et al. generalized a few years ago a classical theorem of Ingham and Beurling by means of divided differences. The optimality of their assumption has been proven by the third author of this note. The purpose of this note to extend these results to vector coefficient sums.
\end{abstract}

\maketitle

\section{Introduction}\label{s1}

Let $\Omega :=(\omega_k)_{k\in \mathbb{Z}}$ be a family of real numbers satisfying the gap condition
\begin{equation}\label{1}
\gamma:=\inf_{k\ne n}|\omega_k-\omega_n|>0.
\end{equation}
Let us denote by $D^+=D^+(\Omega )$ its P\'olya upper density, defined by the formula
$D^{+}:= \lim_{r\to \infty} r^{-1}n^{+}(r)$,
where $n^{+}(r)$ denotes the largest number of terms of the sequence   $(\omega_k)_{k\in \mathbb{Z}}$ contained in an interval of  length  $r$.

Let $(U_k)_{k\in \mathbb{Z}}$ be a corresponding family of unit vectors in some complex Hilbert space $H$ and consider the sums
\begin{equation}\label{2}
x(t)=\sum_{k\in \mathbb{Z}}x_kU_ke^{i\omega_kt}
\end{equation} 
with square summable complex coefficients $x_k$. We are interested in the validity of the estimates
\begin{equation}\label{3}
\int_I|x(t)|_{H}^2\ dt\asymp \sum_{k\in \mathbb{Z}}|x_k|^2
\end{equation} 
where $I$ is a bounded interval of length denoted by $|I|$
and where we write $A\asymp B$ if there exist two positive constants $c_1,c_2$ satisfying $c_1 A\leq B\leq c_2 A.$

We have the following result which generalizes a theorem of Ingham \cite{Ing1936}:

\begin{theorem}\label{t1}\mbox{}

(a) If $|I|>2\pi D^+$, then the estimates \eqref{3} hold true.
\smallskip

(b) If the estimates \eqref{3} hold true and $H$ has a finite dimension $d$, then $|I|\ge 2\pi D^+/d$.
\end{theorem}

For $d=1$ the theorem reduces to the scalar case due to Beurling \cite{Beu}. 

The theorem is sharp in the following sense. Given any real number $\alpha$ between $D^+$ and $D^+/d$, there exists a partition $\Omega =\Omega _1\cup\cdots\cup \Omega _d$ of $\Omega $ such that $\max_{j}D^+(\Omega _j)=\alpha$. Fix an orthonormal basis $E_1,\ldots, E_d$ of $H$ and set $U_k=E_j$ if $\omega_k\in\Omega _j$. Then using the identity
\begin{equation}\label{4}
\int_{I}\Bigl\lVert\sum_{k\in \mathbb{Z}}x_k U_k e^{i\omega_k t}\Bigr\rVert_H^2\ dt
= \sum_{j=1}^{d}\int_I\Bigl|\sum_{\omega_k\in\Omega _j}x_ke^{i\omega_k t}\Bigr|^2\ dt
\end{equation}
and applying the scalar case of the theorem we conclude that the estimates \eqref{3} hold if $|I|>2\pi \alpha$, and they do not hold if $|I|<2\pi \alpha$.

We prove this theorem in the next section and then we extend the result to the case of a weakened gap condition.

We refer to  \cite{KomLor2005} for many control theoretical applications of theorems of this type.

\section{Proof of the Theorem}\label{s2}

Part (a) readily follows from the scalar case. Indeed, fixing an orthonormal basis $(E_n)_{n\in N}$ of the closed linear hull of $(U_k)_{k\in \mathbb{Z}}$ in $H$ and developing the vectors $U_k$ into Fourier series: $U_k= \sum_{n\in N} u_{kn}E_n$,  for $|I|>2\pi D^+$ we have
\begin{equation*}
\int_{I}\Bigl\lVert\sum_{k\in \mathbb{Z}} x_k U_k e^{i\omega_kt}\Bigr\rVert_H^2\  dt
=\sum_{n\in N}\int_I\Bigl|\sum_{k\in \mathbb{Z}} x_k u_{kn}e^{i\omega_kt}\Bigr|^2\ dt
\asymp \sum_{n\in N} \sum_{k\in \mathbb{Z}}|x_k u_{kn}|^2
=\sum_{k\in \mathbb{Z}}|x_k|^2.
\end{equation*}

For the proof of part (b) we adapt the approach developed in \cite{GroRaz1996} and \cite{Meh2005}. We set $\gamma_k:=2\pi |I|^{-1}k$  for brevity.
Given three real numbers $y, r, R$ with $r,R>0$, we introduce the orthogonal projections $P_r:L^2(I,H)\to V_r$ and $Q_{r+R}:L^2(I,H)\to W_{r+R}$ onto the finite-dimensional linear subspaces
\begin{align*}
&V_r:={\rm{\rm Vect}} \left\{U_ke^{i\omega_kt}\ : \  |\omega_k-y|<r\right\}
\intertext{and}
&W_{r+R}:={\rm Vect} \left\{Ue^{i\gamma_nt}\ : \ |\gamma_n-y|<r+R\quad\text{and}\quad U\in H\right\}.
\end{align*}

Setting $S:=P_r\circ Q_{r+R}\circ i$ where $i$ denotes the injection $i:V_r \hookrightarrow L^2(I,H)$, we obtain a linear map of $V_r$ into itself. We are going to study its trace. We denote by 
\begin{equation*}
\Omega _r:= \left\{
\omega_k\ : \ |\omega_k-y|<r,\,\, k\in\mathbb{Z}
\right\}
\quad\text{and}\quad
\Gamma_{r+R}:= \left\{\gamma_k\ : \ |\gamma_k-y|<r+R,\,\, k\in\mathbb{Z}\right\}
\end{equation*}
the sets of exponents figuring in the definition of $V_r$ and $W_{r+R}$.

\begin{lemma}\label{l2}
We have
\begin{equation*}
\left|{\rm tr} (S)\right|\le d\ {\rm Card} (\Gamma_{r+R}).
\end{equation*}
\end{lemma}

\begin{proof}
We have 
\begin{equation*}
\lVert S\rVert\leq\lVert P_r\rVert\cdot\lVert Q_{r+R}\rVert\leq 1.
\end{equation*}
Hence the eigenvalues of $S$ have modulus $\le 1$ and therefore
\begin{equation*}
\left|{\rm tr}(S)\right|\leq {\rm rang}(S)\leq {\rm dim}(W_{r+R}).
\end{equation*}  
Since ${\rm dim}(W_{r+R})= d\ {\rm Card}  (\Gamma_{r+R})$, the lemma follows.
\end{proof}

\begin{lemma}\label{l3}
Writing $e_k(t):= U_ke^{i\omega_kt}$ for brevity, we have
\begin{equation*}
{\rm tr} (S)={\rm Card} (\Omega _r)+\sum_{|\omega_k-y|<r}((Q_{r+R}-{\rm Id})e_k,P_r\varphi_k)_H
\end{equation*}
where $(\varphi_k)_{k\in \mathbb{Z}}$ is a bounded biorthogonal family to $(e_k)_{k\in \mathbb{Z}}$ in $L^2(I,H)$.
\end{lemma}

\begin{proof} We have
\begin{multline*}
{\rm tr}(S)=\sum_{|\omega_k-y|<r}(S e_k,\varphi_k)_{L^2(I,H)}
=\sum_{|\omega_k-y|<r}(Q_{r+R} e_k,P_r\varphi_k)_{L^2(I,H)}\\
=\sum_{|\omega_k-y|<r}( e_k,P_r\varphi_k)_{L^2(I,H)}+\sum_{|\omega_k-y|<r}((Q_{r+R}-{\rm Id}) e_k,P_r\varphi_k)_{L^2(I,H)}.
\end{multline*}
Since $P_r e_k=e_k,$ we have $( e_k,P_r\varphi_k)_{L^2(I,H)}=1$ and the result follows.
\end{proof}

\begin{lemma}\label{l4}
We have
\begin{equation*}
\lVert (Q_{r+R}-{\rm Id} )e_k\rVert=O(1/R)\quad (R\to\infty)
\end{equation*}
uniformly for all $y\in\mathbb{R}$, $r>0$ and $k$ satisfying $|\omega_k-y|<r$.
\end{lemma}

\begin{proof}
Fixing an orthonormal basis $E_1,\ldots, E_d$ of $H$ and setting
\begin{equation*}
f_{n,j}(t):=|I|^{-1/2}E_je^{i\gamma_nt}
\end{equation*}
we have
\begin{align*}
&e_k=\sum_{n\in\mathbb{Z}}\sum_{j=1}^d (e_k,f_{n,j})_{L^2(I,H)}f_{n,j}
\intertext{and}
&Q_{r+R}e_k=\sum_{|\gamma_n-y|<r+R}\sum_{j=1}^d (e_k,f_{n,j})_{L^2(I,H)}f_{n,j}.
\end{align*}
Applying Parseval's equality it follows that
\begin{equation*}
\lVert(Q_{r+R}-{\rm Id} )e_k\rVert^2=\sum_{|\gamma_n-y|\ge r+R}\sum_{j=1}^d\left|(e_k,f_{n,j})_{L^2(I,H)}\right|^2.
\end{equation*}
Since
\begin{equation}
\label{5}
\left|(e_k,f_{n,j})_{L^2(I,H)}\right|=|I|^{-1/2} \Bigl\vert\int_I(U_k,E_j)_H e^{i(\omega_k-\gamma_n)t}\Bigr\vert\le \frac{2|I|^{-1/2}}{
|\omega_k-\gamma_n|},
\end{equation}
and $|\omega_k-y|<r$, then we obtain that
\begin{align*}
\lVert (Q_{r+R}-{\rm Id} )e_k\rVert^2
&\leq 4d|I|^{-1} \sum_{|\gamma_n-y|\ge r+R}\frac{1}{|\omega_k-\gamma_n|^2}\\
&\leq 4d|I|^{-1} \sum_{|\gamma_n-y|>r+R} \frac{1}{||y-\gamma_n|-r|^2}\\
& \le 8d|I|^{-1}\sum_{n=0}^{\infty} \frac{1}{|2\pi|I|^{-1}n+R|^2}.
\end{align*}
Since the last expression doesn't depend on $r,\,\,y$ and is $O(1/R)$ as $R\to\infty,$ the lemma follows.
\end{proof}

Now the proof of part (b) of Theorem \ref{t1} can be completed as follows. By the above lemmas we have
\begin{multline*}
d\ {\rm Card} (\Gamma_{r+R})
\ge |{\rm tr} (S)|
=\Bigl |{\rm Card} (\Omega _r)+\sum_{|\omega_k-y|<r}((Q_{r+R}-{\rm Id})e_k,P_r\varphi_k)_H\Bigr |\\
\ge {\rm Card} (\Omega _r)-O(1/R){\rm Card} (\Gamma_{r+R})
\end{multline*}
and therefore
\begin{equation*}
{\rm Card} (\Omega _r)\le (d+O(1/R)){\rm Card} (\Gamma_{r+R}),\quad R\to\infty.
\end{equation*}
Hence
\begin{equation*}
D^+=\lim_{r\to \infty}\frac{{\rm Card} (\Omega _r)}{r}
\le \lim_{R\to \infty}\lim_{r\to \infty}(d+O(1/R))\frac{{\rm Card} (\Gamma_{r+R})}{r+R}\cdot \frac{r+R}{r}
=\lim_{R\to \infty}(d+O(1/R))\frac{|I|}{2\pi}=\frac{d|I|}{2\pi}
\end{equation*}
and therefore $|I|\ge 2\pi D^+/d$ as claimed.

\section{The case of the divided differences}\label{s3}

The gap condition \eqref{1} of the theorem may be weakened. Following \cite{BaiKomLor2002} let $(\omega_k)_{k\in \mathbb{Z}}$ be a nondecreasing sequence of real numbers satisfying for some positive integer $M$ and for some positive real number $\gamma'$ the weakened gap condition
\begin{equation}\label{6}
\omega_{k+M}-\omega_k\ge M\gamma' \quad\text{for all}\quad k\in \mathbb{Z}.
\end{equation}
This implies that $D^+<\infty$. For $j=1,\dots,M$ and $m \in \mathbb{Z}$
we say that $\omega_m,\dots,\omega_{m+j-1}$ forms a $\gamma'$-close exponent chain if 
\begin{equation*}
\begin{cases}
\omega_m-\omega_{m-1} \ge \gamma', \\
\omega_k-\omega_{k-1} < \gamma' \quad \text{for
  $k=m+1,\dots,m+j-1$}, \\
\omega_{m+j}-\omega_{m+j-1} \ge \gamma'.
\end{cases}
\end{equation*}
Then we define the divided differences  $f_{\ell}=[\omega_{m},\dots,\omega_{\ell}]$ for $\ell=m,\dots,m+j-1$, defined by the formula
\begin{multline*}
[\omega_m,\dots,\omega_{\ell}](t)\,:=
(it)^{{\ell}-1}\int_0^1 \int_0^{s_m} \dots \int_0^{s_{{\ell}-2}}\\
\exp (i[s_{{\ell}-1}(\omega_{\ell}-
\omega_{{\ell}-1})+\dots+s_m(\omega_{m+1}-\omega_m)+\omega_m)]t)\ ds_{{\ell}-1}\dots \ d s_m.
\end{multline*}

We can now state a generalization of Theorem \ref{t1}:

\begin{theorem}\label{t5}\mbox{}
Theorem \ref{t1} holds true if \eqref{1} is replaced by \eqref{6}
and $e^{i\omega_k t}$ is replaced by $f_k(t)$.
\end{theorem}

\begin{proof}
Most of the proof of Theorem \ref{t1} may be easily adapted. For part (b) we have to replace the estimate \eqref{5} by the following:
\begin{equation}\label{7}
\Bigl\vert\int_I(U_k,E_j)_H f_k(t)e^{-i\gamma_nt}dt\Bigr\vert\le \Bigl\vert \int_I f_k(t)e^{-i\gamma_nt}dt\Bigr\vert\le  \frac{C}{
|\omega_k-\gamma_n|},
\end{equation}
with a constant $C$ depending only on $\gamma'$, $M$ and $I$. This is shown by arguing similarly as in \cite{Meh2005}.
We have
\begin{equation*}
A:=\int_I f_k(t)e^{-i\gamma_nt}dt=\int_Ig(t)e^{i\omega_{k} t}e^{-i\gamma_nt}dt
\end{equation*}
with
\begin{equation*}
g(t)=[\omega_m-\omega_{k},\dots,\omega_{k}-\omega_{k}](t).
\end{equation*}
Integrating by parts in $I=(a,b)$ we obtain that
\begin{equation*}
A=\left [\frac{1}{i\omega_{k}-i\gamma_n}g(t)e^{i\omega_{k} t}e^{-i\gamma_nt}\right ]_a^b-\int_I \frac{1}{i\omega_{k}-i\gamma_n} g'(t)e^{i\omega_{k} t}e^{-i\gamma_nt}dt.
\end{equation*}
Now a direct computation shows that for any real numbers $\mu_1,\dots,\mu_r$ the divided differences satisfy the inequality
\begin{equation*}
[\mu_1,\dots,\mu_r]'(t) \le
\frac{(r-1)t^{r-2}}{(r-1)!}+(|\mu_r-\mu_{r-1}|+\dots+|\mu_2-\mu_1|+|\mu_1|)\frac{t^{r-1}}{(r-1)!}.
\end{equation*}
Thus, in our case, thanks to the $\gamma'$-close exponent property,
we have
\begin{equation*}
|g'(t)| 
\le
(k-m)\frac{t^{k-m-1}}{(k-m)!}+(k-m)\gamma'\frac{t^{k-m}}{(k-m)!}
\end{equation*}
and this yields \eqref{7}.
\end{proof}

\end{document}